\PassOptionsToPackage{dvipdfmx}{hyperref}
\documentclass[12pt]{article}
\usepackage{amsmath}
\usepackage{amssymb}
\usepackage{amsthm}
\usepackage{mathrsfs}
\usepackage{authblk}
\usepackage[dvipdfmx]{graphicx}
\usepackage{hyperref}
\usepackage{color}
\usepackage{verbatim}
\usepackage{moreverb}
\usepackage{braket}
\usepackage{bm}
\usepackage{enumitem}
\usepackage{mathtools}
\usepackage{cases}

\mathtoolsset{showonlyrefs}

\newtheorem{prop}{Proposition}[section]
\newtheorem{thm}[prop]{Theorem}
\newtheorem{lem}[prop]{Lemma}
\newtheorem{cor}[prop]{Corollary}
\newtheorem{ddef}[prop]{Definition}
\theoremstyle{remark}
\newtheorem{rem}{Remark}

\DeclareMathAlphabet{\mathpzc}{OT1}{pzc}{m}{it}

\DeclareMathAlphabet{\mathpzc}{OT1}{pzc}{m}{it}
\DeclareMathOperator{\dvg}{div}
\DeclareMathOperator{\grad}{grad}
\DeclareMathOperator{\jgrad}{\mathcal{J}grad}
\DeclareMathOperator{\md}{d}
\DeclareMathOperator{\mj}{\mathcal{J}}
\DeclareMathOperator{\supp}{supp}
\newcommand\pv{\mathop{\mathrm{p.v.}}}
\DeclareMathOperator{\loc}{\mathfrak{L}}
\DeclareMathOperator{\lie}{\mathcal{L}}
\DeclareMathOperator{\curl}{\mathrm{curl}}

\newcommand{\dcal}{\mathcal{D}}
\newcommand{\vf}{\mathfrak{X}}
\newcommand{\dom}{ {\rm Dom}}

\newcommand{\sfs}{ \mathsf{S}}
\newcommand{\oloc}{ \Omega_{[r]\mathrm{loc}}}
\newcommand{\omr}{ \Omega_{[r]}}
 
\newcommand{\re}{{\mathbb R}}
\newcommand{\ce}{{\mathbb C}}
\newcommand{\ze}{{\mathbb Z}}

\newcommand{\mi}{\mathrm{i}}
\newcommand{\eps}{\varepsilon}

\newcommand{\dv}{\mathrm{dVol}_g}

\begin{document}
\title{Mathematical justification of the point vortex dynamics in background fields on surfaces as an Euler-Arnold flow
}
\author{Yuuki Shimizu}
\affil{University of Tokyo}

\maketitle
\abstract{
The point vortex dynamics in background fields on surfaces is justified as an Euler-Arnold flow in the sense of de Rham currents. 
We formulate a current-valued solution of the Euler-Arnold equation with a regular-singular decomposition. 
For the solution, we first prove that, if the singular part of the vorticity is given by a linear combination of delta functions centered at $q_n(t)$ for $n=1,\ldots,N$, $q_n(t)$ is a solution of the point vortex equation. 
Conversely, we next prove that, if $q_n(t)$ is a solution of the point vortex equation for $n=1,\ldots,N$, there exists a current-valued solution of the Euler-Arnold equation with a regular-singular decomposition such that the singular part of the vorticity is given by a linear combination of delta functions centered at $q_n(t)$. 
As a corollary, we generalize the Bernoulli law to the case where the flow field is a curved surface and where the presence of point vortices is taken into account. 
From the viewpoint of the application, the mathematical justification is of a significance since the point vortex dynamics in the rotational vector field on the unit sphere is adapted as a mathematical model of geophysical flow in order to take effect of the Coriolis force on inviscid flows into consideration. 
}

\section{Introduction}
The motion of incompressible and inviscid fluids in the Euclidean plane is governed by the Euler equation, 
\begin{equation}
\label{eq:Euler_equation_intro}
	\partial_t v_t + (v_t\cdot \nabla) v_t =-\grad p_t,\quad \dvg v_t= 0, 
\end{equation}
where $(v_t\cdot \nabla) = v_t^1\partial_1+v_t^2\partial_2 $. 
The fluid velocity field $v_t\in \vf( \re^2 )$ and the scalar pressure $p_t\in C( \re^2)$ are the unknowns at time $t\in (0,T]$.
A solution of the Euler equation~\eqref{eq:Euler_equation_intro} is called an Euler flow. 
Applying the curl operator to the Euler equation, we obtain the following equation for the vorticity $\omega_t=\mathop{\rm curl}v_t = \partial_1 v^2_t -\partial_2v^1_t$.
\begin{equation}
\label{eq:vorticity_equation_plane_intro}
	\partial_t \omega_t + (v_t\cdot \nabla) \omega_t = 0,
\end{equation}
which is called the vorticity equation. 
Owing to $\dvg v_t=0$, the velocity becomes a Hamiltonian vector field, i.e., $v_t=-\jgrad \psi_t=(\partial_2 \psi_t,-\partial_1\psi_t)$ for some $\psi_t\in C(\re^2)$, where $\mathcal{J}$ is the symplectic matrix. 
The unknowns $\psi_t$, $v_t$ and $p_t$ are presented by $\omega_t$. 
Since $\psi_t$ satisfies  
\begin{equation}
	\omega_t =-\curl \jgrad \psi_t =-(\partial_1^2+\partial_2^2)\psi_t=-\triangle\psi_t, 
\end{equation}
it is given by $\psi_t=\langle G_H,\omega_t\rangle$ and $v_t= -\jgrad \langle G_H, \omega_t\rangle$, where $G_H$ is the Green's function for the Laplacian $-\triangle$. 
To determine $p_t$ from $v_t$, Applying the divergence operator to the equation~\eqref{eq:Euler_equation_intro}, we have
\begin{equation}
	\dvg (v_t\cdot \nabla) v_t=-\dvg \grad p_t=-\triangle p_t.
\end{equation}
Hence, the pressure is determined as a solution of this Poisson equation. 
We thus obtain 
\begin{equation}
\label{eq:formal_velocity_pressure}
	v_t= -\jgrad \langle G_H, \omega_t\rangle,\quad p_t = \langle G_H,\dvg (v_t\cdot \nabla) v_t\rangle.
\end{equation} 
Note that it is derived under the assumption that $(v_t,p_t)$ is an Euler flow. 

On the other hand, the formulae~\eqref{eq:formal_velocity_pressure} still make sense in the sense of distributions when we give a time-dependent distribution $\Omega_t= \sum_{n=1}^N \Gamma_n \delta_{q_n(t)}$ as a linear combination of delta functions centered at $q_n$ for $n=1,\ldots,N$. 
Replacing $\omega_t$ by $\Omega_t$ in the equation~\eqref{eq:formal_velocity_pressure}, we formally obtain a velocity field $V_t\in \vf(\re^2\setminus\{q_n(t)\}_{n=1}^N)$ and a pressure $\Pi_t\in C(\re^2\setminus\{q_n(t)\}_{n=1}^N)$. 
However, the pair $(V_t,\Pi_t)$ is no longer an Euler flow in a regular sense, since it does not belong to $L^2_{\rm loc}(\re^2)$. 
Hence, we can not define the dynamics of $q_n(t)$ from the Euler equation. 
Instead, to determine the evolution of $q_n(t)$ by $V_t$, Helmholtz considered the following regularized equation for $q_n(t)$~\cite{Saffman_1992}.
\begin{align}
	\label{eq:point_vortex_equation_intro}
	\dot{q}_n
	&=\lim_{q\to q_n} \left[ V_t(q)+\jgrad \langle G_H, \Gamma_n \delta_{q_n(t)}\rangle(q)\right] \nonumber \\
	&=-\jgrad \sum_{\substack{m=1\\ m\ne n}}^N \Gamma_m G(q_n,q_m)\equiv v_n(q_n) .	
\end{align}
It is called the \textit{point vortex equation}, and the solution of~\eqref{eq:point_vortex_equation_intro} is called the \textit{point vortex dynamics}. 
Then, there arises a natural question; 
How can we interpret $(V_t,\Pi_t)$ as an Euler flow in an appropriate mathematical sense? 
In other words, we need to determine a space of solutions of the Euler equation which contains $(V_t,\Pi_t)$. 
Since $L^2_{\rm loc}(\re^2)$ can not be such a space of solutions, a more sophisticated space is to be considered.
Due to the singularity of the vorticity and the nonlinearity of the Euler equation, this mathematical justification is an important but difficult problem in the analysis of the two-dimensional Euler equation.
Starting from the work by Marchioro and Pulvirent~\cite{Marchioro_Pulvirenti_1983}, there have been many studies on the indirect characterization of the point vortex dynamics by approximate sequences of weak solutions of the Euler equation~\cite{Majda_Bertozzi_2002,Marchioro_Pulvirenti_1994}.
On the other hand, this problem has remained open in terms of constructing a space of solutions that directly includes the point vortex dynamics.

The point vortex dynamics is sometimes considered in the presence of the velocity field $X_t\in \vf^r(\re^2)$ of an Euler flow, in which the evolution of $q_n(t)$ is governed by the following equation. 
\begin{equation}
\label{eq:point_vortex_equation_background_intro}
	\dot q_n(t)= \beta_X X_t (q_n(t))+\beta_\omega v_n(q_n(t)), \quad n=1,\ldots N 
\end{equation}
for a given parameter $(\beta_X,\beta_\omega)\in\re^2 $. 
The velocity field $X_t$ is called the background field.
An experimental study confirm the importance of background fields in two-dimensional turbulence~\cite{Trieling_Dam_Heijst_2010}. 

In the same way as in the absence of a background field, a velocity field $V_t\in \vf(\re^2\setminus\{q_n(t)\}_{n=1}^N)$ and a pressure $\Pi_t\in C(\re^2\setminus\{q_n(t)\}_{n=1}^N)$ are defined: 
\begin{equation}
\label{eq:V_t_P_t_generalization}
	{V}_t(q)=X_t(q)-\jgrad \sum_{n=1}^N \Gamma_n G_H(q,q_n(t)), \quad  {P}_t=\langle G_H,\dvg ({V}_t\cdotp \nabla) {V}_t\rangle. 
\end{equation}
Even if the regularity of the background field is simply of class $C^r$, it remains open whether the pair is a weak solution of the Euler equation as it is open in the case with no background field.

The purpose of this paper is justifying the point vortex dynamics in background fields as an Euler flow mathematically. 
To this end, we establish a weak formulation of the Euler equation in the space of currents, which is developed in the theory of geometric analysis and geometric measure theory. 
The notion of currents is defined not only for the Euclidean plane but also general curved surface. 
Hence, the formulation established here can be naturally generalized for surfaces. 
The Euler equation is generalized for the case of surfaces by Arnold~\cite{Arnold_1966}, called the Euler-Arnold equation. 
From the viewpoint of the application, it is of a significance to justify the point vortex dynamics in a background field on curved surfaces as an Euler-Arnold flow, since the point vortex dynamics in the rotational vector field on the unit sphere is adapted as a mathematical model of a geophysical flow in order to take effect of the Coriolis force on inviscid flows into consideration~\cite{Newton_Shokraneh_2006} for instance. 
As for prerequisites, it is helpful to be familiar with the theory of manifolds and calculus of differential forms, although some basic notions such as vector calculus on curved surfaces related with fluid dynamics, de Rham currents, and the Euler-Arnold equation are reviewed in this paper.

This paper is organized as follows. 
In Section~\ref{sec:backgound}, we collect some definitions and conventions from differential geometry that will be used throughout the paper.  
In Section~\ref{sec:point_vortex_dynamics}, we derive the point vortex dynamics on curved surfaces from the Euler-Arnold equation in the similar manner to the case of the Euclidean plane.  
In Section~\ref{sec:current}, we review basic concepts of the theory of de Rham currents. 
Some notions in vector calculus are reformulated in terms of currents for the application to fluid dynamics. 
In Section~\ref{sec:Euler_Arnold}, we examine the Euler-Arnold equation for more details from the viewpoint of currents.  
We formulate a current-valued solution of the Euler-Arnold equation on surfaces with a regular-singular decomposition, which we call a \textit{$C^r$-decomposable weak Euler-Arnold flow}.  
In Section~\ref{sec:main_results}, our main results are stated and proved. 
In the first theorem, we prove that, for a given $C^r$-decomposable weak Euler-Arnold flow, if the singular part of the vorticity is given by a linear combination of delta functions centered at $q_n(t)$ for $n=1,\ldots,N$, $q_n(t)$ is a solution of the point vortex equation~\eqref{eq:point_vortex_equation_background_intro}.  
Conversely, in the second theorem, we next prove that, if $q_n(t)$ is a solution of the point vortex equation~\eqref{eq:point_vortex_equation_background_intro} for $n=1,\ldots,N$, there exists a $C^r$-decomposable weak Euler-Arnold flow such that the singular part of the vorticity is given by a linear combination of delta functions centered at $q_n(t)$.
As a consequence, the point vortex dynamics in a background field on a surface is justified as a $C^r$-decomposable weak Euler-Arnold flow. 
In Section~\ref{sec:discussion}, we apply the main results to some problems in fluid dynamics. 
As a corollary, we generalize the Bernoulli law to the case where the flow field is a curved surface and where the presence of point vortices is taken into account. 

\section{Background materials} 
\label{sec:backgound}
Let $(M,g)$ be a connected orientable $2$-dimensional Riemannian manifold, called a surface, with or without boundary, which can be compact or non-compact. 
Here, $g$ denotes the Riemannian metric. 
Let $\vf^r(M)$ be the space of all vector fields on $M$ of class $C^r$ and $\Omega_{[r]}^p(M)$ be the space of all $p$-forms on $M$ of class $C^r$. 
We denote $\Omega_{[\infty ]}^p(M)$ briefly by $\Omega^p(M)$. 
Let $\dcal^p(M)$ be the space of all smooth $p$-forms with compact support. 

The Riemannian metric defines the musical isomorphism between $\vf^r(M)$ and $\omr^1(M)$. 
Every vector field is converted to a $1$-form by the flat operator $\flat: X\in \vf^r(M)\to X^\flat=g(X,\cdotp)\in \omr^1(M)$. 
The inverse of the flat operator is called the sharp operator 
$\sharp :\alpha \in \omr^1 (M) \to \alpha_\sharp \in \vf^r(M)$. 
Then, the $1$-form $X^\flat$ and the vector field $\alpha_\sharp$ is called the velocity form and the dual vector field. 
Moreover, the Riemannian metric $g$, which is defined on the tangent bundle is extended to the metric $\tilde g$ on the cotangent bundle such that $\tilde g(\alpha,\beta)=g(\alpha_\sharp,\beta_\sharp)$ for all $\alpha,\beta\in \Omega_{[r]}^1(M)$. 
For simplicity of notation, we use the same letter $g$ for $\tilde g$. 

The Hodge-$*$ operator is defined as $\alpha\wedge *\beta =g(\alpha,\beta)\dv$ for all $\alpha,\beta\in \Omega_{[r]}^p(M)$,  
where $\dv\in \Omega^2(M)$ denotes the Riemannian volume form. 
Owing to the orientability of $M$, the Hodge-$*$ operator induces a complex structure $\mathcal{J}: TM \to TM$ which satisfies $\mathcal{J} X=(*X^\flat)_\sharp$ for each $X\in\vf^r(M)$.  
In particular, $\mathcal{J}$ rotates vectors by the degree $+\pi/2$. 

Some notions in vector calculus such as the divergence, the curl, the gradient and the Hamiltonian vector field are rewritten in terms of differential forms. 
Let $\md $ denote the differential operator $\md: \Omega_{[r]}^p(M)\to \Omega_{[r-1]}^{p+1}(M)$. 
The codifferential operator is defined by $\delta=*\md *: \omr^p(M)\to \Omega_{[r-1]}^{p-1}(M)$. 
We denote the Hodge Laplacian by $\triangle=\md \delta+\delta\md :\omr^p(M)\to \Omega_{[r-2]}^p(M)$. 
Let us define the divergence operator $\dvg: \vf^r(M)\to \Omega_{[r-1]}^0(M)$ and the curl operator $\curl: \vf^r(M)\to \Omega_{[r-1]}^0(M)$ by 
\begin{equation}
	\dvg X=\delta X^\flat,\quad \curl X =*\md X^\flat.
\end{equation}
Note that $\omr^0(M)=C^r(M)$. 
A vector field $X\in\vf^r(M)$ is said to be incompressible (or divergence free) if $\dvg X=0$. 
Given $X\in\vf^r(M)$, we call the function $\omega=\curl X$ the vorticity. 
Let us define the gradient operator $\grad: \omr^0(M)\to \vf^{r-1}(M)$ by 
\begin{equation}
	\grad \phi =(\md \phi)_\sharp. 
\end{equation}
A vector field $X\in\vf^r(M)$ is called a gradient vector field if the velocity form $X^\flat$ is an exact $1$-form, or equivalently, there exists a function $\phi \in \omr^0(M)$ such that $X^\flat=\md \phi$,  i.e., $X=\grad \phi$.
In contrast, a vector field $X\in\vf^r(M)$ is called a Hamiltonian vector field if the velocity form $X^\flat$ is a coexact $1$-form, or equivalently, there exists a function $\psi\in  \omr^0(M)$ such that $X^\flat =-* \md \psi$, i.e., $X=-\jgrad \psi$. 
Here, the symplectic form on the surface is taken as the Riemannian volume form. 
By definition, we obtain 
\begin{align}
		\curl \jgrad\psi_t &= *\md (\jgrad\psi_t)^\flat=*\md *\md \psi_t=(\delta\md +\md \delta)\psi_t=\triangle\psi_t, \\
		\dvg \grad \phi_t &= *\md *(\grad \phi_t)^\flat=*\md *\md \phi_t =(\delta\md +\md \delta)\phi_t=\triangle \phi_t.
\end{align}
Hence, the vorticity $\omega$ of $X=-\jgrad\psi$ satisfies 
\begin{align}
\label{eq:Poisson_equation_surface}
	\omega=-\triangle \psi,
\end{align}
which yields that $\psi$ is determined by a solution of the Poisson equation~\eqref{eq:Poisson_equation_surface}. 
Owing to the slip boundary condition for $X$, $\psi$ obeys the Dirichlet boundary condition with a constant boundary value, since 
\begin{align}
	\md \psi =*X^\flat=0\quad\text{on }\partial M.
\end{align}
We define a function $G_H\in C^\infty(M\times M \setminus \Delta)$ as the fundamental solution of the Poisson problem and call it the \textit{hydrodynamic Green function}. 
\begin{ddef}
\label{def:Green}
	A function $G_H\in C^\infty(M\times M \setminus \Delta)$ is called a \textit{hydrodynamic Green function}, if the function $G_H$ is a solution of the following boundary value problem. 
	For each $(x,x_0)\in M\times M\setminus \Delta$ and each $\phi\in\dcal^2(M)$, 
\begin{align}
	-\triangle G_H(x,x_0)[\phi] &=
	\begin{dcases}
		*\phi(x_0) -\frac{1}{{\rm Area}(M)}\int_M\phi,\quad &\text{if $M$ is closed},\\
		*\phi(x_0),\quad &\text{otherwise},
	\end{dcases}\\
	G_H(x,x_0)&=G_H(x_0,x),\\
	\md G_H &=0\quad\text{on }\partial M.
\end{align} 
\end{ddef}
A regularized hydrodynamic Green's function, called Robin function $R\in C^\infty(M)$~\cite{Flucher_1999,Ragazzo_Viglioni_2017},  is defined by 
\begin{equation}
	R(x)=\lim_{x_0\to x} G_H(x,x_0)+\frac{1}{2\pi} \log d(x,x_0),
\end{equation}
where $d\in C^\infty(M\times M)$ is the geodesic distance on $M$. 

We note that the boundary value of $G_H$ is arbitrarily fixed as long as the summation of the flux on each boundary component is equals to $1$ and that the difference of this choice is at most harmonic functions, say $\psi_0$. 
There are several definitions of the hydrodynamic Green's function according to the boundary value and asymptotic behavior near ends of surfaces~\cite{Flucher_1999,Ragazzo_Viglioni_2017}. 
By solving the Poisson problem~\eqref{eq:Poisson_equation}, we obtain $\psi= \psi^0 +\langle G_H,\omega\rangle$ for some harmonic function $\psi^0$, which gives the Biot-Savart law on the surface. 
\begin{equation}
\label{eq:X_omega_surface}
	X= -\jgrad \psi=-\jgrad(\psi_0+ \langle G_H,\omega\rangle).
\end{equation}

As conventions, we assume that for the surface $(M,g)$, there exists a hydrodynamic Green function $G_H$ for the Hodge Laplacian $\triangle$ and arbitrarily choose a hydrodynamic Green function $G_H$. 
For instance, if the surface is compact, the existence is guaranteed in~\cite{Aubin_1998}. 
Otherwise, the existence is discussed from the viewpoint of complex analysis~\cite{Sario_Nakai_1970}. 
In addition, we assume that if $M$ is a surface with boundary, every vector field $X\in \vf^r(M)$ satisfies the slip boundary condition on $\partial M$, that is $X|_{\partial M}\in \vf^r(\partial M) $.
Note that the slip boundary condition is also written as $*X^\flat =0$ in $\omr^1(\partial M)$. 

Let us check the above notions are consistent with the notions in the case where the surface is given as the Euclidean plane $(\re^2, (\md x^1)^2+(\md x^2)^2)$. 
Every vector field $X\in \vf^r(\re^2)$ is conventionally written as $X={}^{\mathsf{T}}\!(X^1,X^2)$. 
However, in the theory of manifolds, it is rewritten as $X= X^1\partial_1 +X^2\partial_2$ by using an orthogonal basis $(\partial_1,\partial_2)$ of the tangent space of $\re^2$. 
Here, the orthogonal basis is identified with the partial derivative. 
Writing the dual basis in the cotangent space of $\re^2$ by $\langle \md x^1, \md x^2\rangle$, we obtain the velocity form $X^\flat= X^1\md x^1+X^2 \md x^2 \in \omr^1(\re^2)$ for each $X= X^1\partial_1 +X^2\partial_2 \in\vf^r(\re^2)$ and the dual vector field $\alpha_\sharp = \alpha_1\partial_1 +\alpha_2\partial_2 \in\vf^r(\re^2)$ for each $\alpha = \alpha_1\md x^1+\alpha_2 \md x^2 \in \omr^1(\re^2)$. 
Obviously, the Euclidean metric, which is $(X,Y)= X^1 Y^1+X^2Y^2$ for each $X,Y\in\vf^r(\re^2)$, is extended to the metric on the cotangent bundle such that $(\alpha ,\beta )= \alpha^1 \beta^1+\alpha^2\beta^2$ for each $\alpha,\beta\in\omr^1(\re^2)$. 
For each $p$-form $\alpha\in \Omega_{[r]}^p(\re^2)$, $p=0,1,2$, $\md \alpha\in \Omega_{[r-1]}^{p+1}(\re^2)$ and $* \alpha\in \Omega_{[r]}^{2-p}(\re^2)$ satisfy
\begin{align}
	\md \alpha &=
	\begin{cases}
	\partial_1 \alpha \md x^1+\partial_2 \alpha \md x^2,&\quad\text{if }p=0,\\
	(\partial_1 \alpha_2 -\partial_2 \alpha_1 ) \md x^1\wedge\md x^2,&\quad\text{if }p=1,\\
	0,&\quad\text{if }p=2,
	\end{cases}\\
	* \alpha &=
	\begin{cases}
	\alpha_0 \md x^1\wedge\md x^2,&\quad\text{if }p=0,\\
	-\alpha_2 \md x^1+\alpha_1 \md x^2,&\quad\text{if }p=1,\\
	\alpha_{12},&\quad\text{if }p=2.
	\end{cases}
\end{align}
Hence, the complex structure $\mj: T\re^2\to T\re^2$ is exactly the symplectic matrix since
\begin{align}
	\mj\partial_1=\partial_2,\quad \mj\partial_2=-\partial_1. 
\end{align}
Moreover, we obtain 
\begin{align}
	\dvg X
	&= \delta X^\flat
	= *\md *(X^1 \md x^1+X^2 \md x^2)\\
	&= *\md (-X^2 \md x^1+X^1 \md x^2)\\
	&= * (\partial_1 X^1+\partial_2 X^2) \md x^1\wedge \md x^2\\
	&= \partial_1 X^1 +\partial_2 X^2,\\
	\curl X
	&=*\md  X^\flat
	= *\md (X^1 \md x^1+X^2 \md x^2)\\
	&= * (\partial_1 X^2-\partial_2 X^1) \md x^1\wedge\md x^2\\
	&= \partial_1 X^2 -\partial_2 X^1.
\end{align}
In the same manner, we have 
\begin{align}
	\grad \phi 
	&=(\md  \phi)_\sharp\\
	&= (\partial_1 \phi \md x^1 +\partial_2 \phi \md x^2)_\sharp\\
	&= \partial_1 \phi \partial_1 +\partial_2 \phi \partial_2,\\
	-\jgrad\psi
	&=-(*\md  \psi)_\sharp	 \\
	&= (\partial_2 \psi \md x^1 -\partial_1 \psi \md x^2)_\sharp\\
	&= \partial_2 \psi \partial_ 1 -\partial_1 \psi \partial_2.
\end{align}

\section{Point vortex dynamics on surfaces}
\label{sec:point_vortex_dynamics} 
Also for curved surfaces, the point vortex dynamics is derived from the Euler-Arnold equation in a similar manner as the case of the Euclidean plane.
The Euler-Arnold equation on a surface $(M,g)$ is written as 
\begin{equation}
\label{eq:Euler_Arnold_equation_intro}
	\partial_t v_t + \nabla_{v_t} v_t =-\grad p_t,\quad \dvg v_t= 0.   
\end{equation}
The fluid velocity $v_t\in \vf(M)$ and the pressure $p_t\in C(M)$ are the unknowns at time $t\in(0,T]$. 
The fluid velocity satisfies the slip boundary condition. 
The solution $(v_t,p_t)$ is called an Euler-Arnold flow. 
The references to the Euler-Arnold equation are found in~\cite{Arnold_1966,Arnold_Khesin_1998,Ebin_Marsden_1970,Taylor_2011}.  
The advection term $(v_t\cdot \nabla) v_t $ in the Euler equation~\eqref{eq:Euler_equation_intro} is replaced as $\nabla_{v_t} v_t$ in the Euler-Arnold equation~\eqref{eq:Euler_Arnold_equation_intro} by using the Levi-Civita connection $\nabla$. 

Applying the curl operator $\curl =*\md $ to the equation~\eqref{eq:Euler_Arnold_equation_intro} and writing the Lie derivative by ${\cal L}$, we deduce from $\nabla_{v_t}v_t^\flat = {\cal L}_{v_t} v_t^\flat -\md |v|^2/2$ that 
\begin{equation}
\label{eq:vorticity_equation_surface_intro}
	\partial_t \omega_t +{\cal L}_{v_t}\omega_t=0.
\end{equation}
The equation~\eqref{eq:vorticity_equation_surface_intro} is the  vorticity equation on the surface. 
As we see in the case of the Euclidean plane, we present $v_t$ and $p_t$ by $\omega_t$. 
Since $v_t$ is not always a Hamiltonian vector field, we assume that $v_t$ is a Hamiltonian vector field, thereby deriving the presentation~\eqref{eq:formal_velocity_pressure} for the surface. 
Then $v_t$ is written as $v_t=-\jgrad\psi_t $ for some $\psi_t\in C(M)$. 
Hence, we obtain 
\begin{equation}
	\omega_t=-\curl\jgrad\psi_t=-\triangle\psi_t.
\end{equation} 
By the Green function $G_H$ for the Hodge Laplacian $-\triangle$, $\psi_t$ is given by $\psi_t =\langle G_H,\omega_t\rangle$, which yields $v_t= -\jgrad\langle G_H,\omega_t\rangle$. 
Similarly, applying the divergence operator to the equation~\eqref{eq:Euler_Arnold_equation_intro}, we obtain 
\begin{equation}
	\dvg \nabla_{v_t} v_t =-\dvg\grad p_t=-\triangle p_t. 
\end{equation} 
Hence, under the assumption that $(v_t,p_t)$ is an Euler-Arnold flow and $v_t$ is a Hamiltonian vector field, we deduce that  the velocity field $v_t$ and the pressure $p_t$ are given by 
\begin{equation}
\label{eq:formal_velocity_pressure_surface}
		v_t= -\jgrad \langle G_H, \omega_t\rangle,\quad p_t = \langle G_H,\dvg \nabla_{v_t} v_t\rangle. 
\end{equation}

As the time-dependent distribution, taking a linear combination of delta functions $\Omega_t =\sum_{n=1}^N \Gamma_n \delta_{q_n(t)}$ centered at $q_n(t)\in M$, we obtain a time-dependent vector field $V_t\in \vf(M\setminus\{q_n(t)\}_{n=1}^N)$ and a time-dependent function $\Pi_t\in C (M\setminus\{q_n(t)\}_{n=1}^N)$ which is defined by~\eqref{eq:formal_velocity_pressure_surface} in the sense of distributions. 
As an analogy of the Helmholtz's principle, let us consider the following regularized equation for $q_n(t)$. 
\begin{align}
	\label{eq:point_vortex_equation_intro_surface}
	\dot{q}_n
	&=\lim_{q\to q_n} \left[ V_t(q)+\jgrad \Gamma_n \log d(q,q_n(t))\right] \\
	&=-\jgrad_{q_n} \sum_{\substack{m=1\\ m\ne n}}^N \Gamma_m G(q_n,q_m) + \Gamma_n R(q_n) \equiv v_n(q_n), 
\end{align}
where $d$ is the geodesic distance. 
The equation~\eqref{eq:point_vortex_equation_intro_surface} is also called the point vortex equation, and the solution is called the point vortex dynamics on the surface.

Point vortex dynamics on surfaces is originally motivated by the applications to geophysical fluids~\cite{Bogomolov_1979}. 
In addition to the above derivation, which is based on the analogy of Helmholtz's principle, there are several other derivations of the point vortex dynamics on surfaces. 
The point vortex dynamics is investigated in the many surfaces: a sphere~\cite{Kimura_Okamoto_1987}, a hyperbolic disc~\cite{Kimura_1999}, multiply connected domains~\cite{Sakajo_2009}, a cylinder~\cite{Montaldi_Souliere_Tokieda_2003}, a flat torus~\cite{Tkachenko_1966}, the Bolza surface~\cite{Ragazzo_2017}, surfaces of revolution which is diffeomorphic to the plane~\cite{Hally_1980}, the sphere~\cite{Dritschel_Boatto_2016} and the torus~\cite{Sakajo_Shimizu_2016}. 
For general curved surfaces, another derivation of point vortex dynamics from the vorticity equation~\eqref{eq:vorticity_equation_surface_intro} and the generalized Newton law has been recently developed in~\cite{Ragazzo_Viglioni_2017}.
As is the case of the Euclidean plane, we can ask whether $(V_t,\Pi_t)$ is an Euler-Arnold flow. 
However, the problem is still open for curved surfaces as well as the plane. 
Besides the problem as we see in the case of the Euclidean plane, due to the generalization of point vortex dynamics to the curved surfaces, the other problems arise in justifying the point vortex dynamics as an Euler-Arnold flow. 
First, the problem comes from the generalization of Helmholtz's principle to the surface. 
In the analogy of the Helmholtz's principle, the fluid velocity is regularized by the geodesic distance. 
There is no particular reason why the geodesic distance should be taken to regularize the fluid velocity. 
Another choice of the regularizing function has been recently discussed in~\cite{Gustafsson_2019}. 
Second, there exists an Euler-Arnold flow such that it is not Hamiltonian owing to the Hodge decomposition, whereas the point vortex dynamics is formulated as a Hamiltonian dynamics. 
Hence, there is a gap between non-Hamiltonian Euler-Arnold flows and the point vortex dynamics. 

Let us fill this gap by taking the background field surrounding the point vortex dynamics. 
The background field is now taken as a classical solution of the Euler-Arnold equation. 
Let us define the point vortex dynamics on surfaces in a background field by a solution of the following equation. 
\begin{equation}
\label{eq:point_vortex_equation_background_intro_surface}
	\dot q_n(t)= \beta_X X_t (q_n(t))+\beta_\omega v_n(q_n(t)), \quad n=1,\ldots N,
\end{equation}
for a given $(\beta_X,\beta_\omega)\in\re^2 $. 
For example, in the application to geophysical flows, the point vortex dynamics in a background field on the unit sphere is adapted as a mathematical model of incompressible and inviscid fluid flows on the unit sphere with Coriolis force~\cite{Newton_Shokraneh_2006}. 
On the other hand, the problem remains open as to whether the point vortex dynamics in a background field is an Euler-Arnold flow.

\section{de Rham current}
\label{sec:current}
\subsection{Basic concepts and properties}
\label{sec:current_basic}
We review some basic notions in the theory of de Rham currents, which are alternatives of the Schwartz distributions in the Euclidean space that are required to give a weak formulation of the Euler-Arnold equation on curved surfaces. 
Roughly speaking, currents are differential forms with distribution coefficients in local charts.  
A good reference of the theory of de Rham currents is given in~\cite{deRham_1984}. 

A $p$-current is defined as a continuous linear functional over $\re$ on ${\cal D}^{2-p}(M)$. 
Let $T[\phi]$ denote the coupling of the $p$-current $T$ and $\phi\in{\cal D}^{2-p}(M)$.
The elements of ${\cal D}^p(M)$ are often called test forms. 
The space of all $p$-currents on $M$ is denoted by ${\cal D}_p'(M)$. 
The calculus of differential forms such as the differential operator $\md$, the Hodge-$*$ operator, the codifferential operator $\delta=*\md *$, Hodge Laplacian $\triangle=\md \delta+\delta\md$ can be extended to currents via test forms.
In fact, for a given $T\in {\cal D}_p'(M)$, 
$\md$ and $*$ are defined by $\md T[\phi]=(-1)^{p+1}T[\md\phi]$ for each $\phi\in{\cal D}^{1-p}(M)$ and $*T[\varphi]=(-1)^{p(2-p)} T[*\varphi]$ for each $\varphi\in \dcal^{p}(M)$. 
Thus the notions for differential forms such as (co)closedness, (co)exactness and harmonicity can also be defined with respect to currents. 

For instance, the space $\dcal_0'(M)$ corresponds to the space of  distributions on $M$.
For any $p$-form $\alpha\in \Omega^p(M)$, it is naturally identified with a $p$-current when we define a functional $I(\alpha)$ on ${\cal D}^{2-p}(M)$ as $I(\alpha)[\phi]=\int_M\alpha\wedge\phi$. 
For $p\in M$, we define \textit{the delta current}, say $\delta_p\in {\cal D}_0'(M)$, by $\delta_p[\phi]=*\phi(p)$. 
This is the counterpart of the delta function in the theory of distributions. 

We now introduce $\chi_p: v\in\vf^r(U)\to \chi_pv \in{\cal D}_{n-1} '(M)$ for a given open subset $U\subset M$ and $p\in U$ by $\chi_pv[\phi]=\phi_p(v_p)$ for each $\phi\in\dcal^1(M)$.
It is characterized as a limit of the mean value for the vector field around a geodesic circle as follows. 
\begin{prop}
\label{prop:mean_value_theorem} 
Fix an open subset $U \subset M$, $p\in U$ and $v\in \vf^r(U)$. 
Then for any $\phi \in{\cal D}^1(M)$, we have 
\begin{equation}
\chi_pv[\phi]=\lim_{\varepsilon\rightarrow 0}\int_{\partial B_{\varepsilon}(p)}\frac{1}{\pi}g(v,\jgrad \log d(p,q)) \phi. 
\end{equation}

\end{prop}
\begin{proof}
	Let us take a complex coordinate $(z)$ centered at $p$ with $z(p)=0$. 
	Since $g$ is presented as $g=\lambda^2|\md z|^2$ for some $\lambda\in C^\infty(U)$, 
	$d(p,q)$ is written as $d(p,q)=\lambda(0)|z|+O(\eps)$. 
	It follows from $*\md z=-\mi\md z$ and $*\md \bar{z}=\mi\md \bar{z}$ that 
	\begin{align}
		g(v,{\rm sgrad}\log d(p,q))
		&= *\md \log d(p,q) [v]
		= *\md \log |z| [v] + O(\eps)\\
		&= \frac{\mi}{2}\left(-\frac{\md z}{z}+\frac{\md \bar{z}}{\bar{z}} \right)[v_z \partial_z+v_{\bar{z}} \partial_{\bar{z}}] + O(\eps)\\
		&= \frac{\mi}{2}\left(-\frac{v_z}{z}+\frac{v_{\bar{z}}}{\bar{z}} \right) + O(\eps).
	\end{align}
	By the residue theorem, we conclude that 
	\begin{align}
		\int_{\partial B_{\varepsilon}(p)}\frac{1}{\pi}g(v,{\rm sgrad}\log d(p,q)) \phi
		&=\int_{\partial B_{\varepsilon}(p)}\frac{\mi}{2\pi}\left(-\frac{v_z}{z}+\frac{v_{\bar{z}}}{\bar{z}} \right)(\phi_z\md z+\phi_{\bar{z}}\md \bar{z})+ O(\eps)\\
		&=v_z(0)\phi_z(0)+v_{\bar{z}}(0)\phi_{\bar{z}}(0)+ O(\eps)\\
		&\to \chi_pv[\phi].
	\end{align}
\end{proof}

Let us formulate a notion of currents corresponding to singular integral operators. 
We first establish the space of all singular integral kernels. 
Namely, we call a $p$-form on a nonempty open subset $U\subset M$ \textit{local $p$-form}. 
Let $\Omega_{[r]{\rm loc}}^p(M)$ denote the space of all  local $p$-form.
In the same manner as $L^p_{\rm loc}(M)$, the topology of $\Omega_{[r]{\rm loc}}^p(M)$ can be defined. 
For each $\alpha\in \Omega_{[r]{\rm loc}}^p(M)$, there exists a maximal open subset $U$ such that $\alpha\in \Omega_{[r]}^p(U)$. 
Then, the closed subset $\sfs^r(\alpha)=M\setminus U$ is called the singular support of $\alpha$. 
We next define a singular integral operator as an integral operator whose integral kernel is local $p$-form. 
Namely, a $p$-current $T\in \mathcal{D}_p'(M)$ is said to be of class $C^r$, if there exists a local $p$-form $\alpha_T\in \Omega_{[r]{\rm loc}}^p(M)$ such that for every $\phi\in\dcal^{2-p}(M\setminus\sfs^r(\alpha_T) )$, $T[\phi]=I(\alpha_T)[\phi]$. 
Let $\mathcal{D}_p'^r(M)$ denote the space of all $p$-currents of class $C^r$.
For each $T\in \mathcal{D}_p'^r(M)$, the subset $\sfs^r(T)=\sfs^r(\alpha_T)$ and the local $p$-form $\alpha_T$ is called the singular support of $T$ and the density of $T$. 
Owing to the fundamental lemma of calculus of variation, the density is uniquely determined. 
Thus, the map $K: T\in\dcal_p'^r(M)\to K(T)=\alpha_T \in \Omega_{[r]{\rm loc}}^p(M)$ is well-defined and called the derivative. 
For example, for each $p\in M$, the delta current $\delta_p$ is $C^\infty$ since $\delta_p=I(0)$ in $M\setminus\{p\}$. 
We thus obtain $\mathsf {S}^\infty ( \delta_p )=\{p\}$ and $K(\delta_p)=0$.  
In this paper, all currents are $C^r$ and the singular support consists of a finite set of points. 

\begin{rem}
\label{rem:Green}
	In Definition~\ref{def:Green}, we state the definition of the hydrodynamic Green function in the sense of distribution. 
Let us restate the definition in the sense of currents. 
For each $x_0\in M$, we define $G_{x_0}\in \dcal_0'^\infty(M)$ as $K(G_{x_0})(x)=G_H(x,x_0)$. 
Then, the definition of the hydrodynamic Green function $G_H\in C^\infty(M\times M\setminus\Delta)$ is rewritten in terms of the current $G_{x_0}$ as follows. 
For each $(x,x_0)\in M\times M\setminus \Delta$ and each $\phi\in\dcal^2(M)$, 
\begin{align}
	-\triangle G_{x_0}[\phi] &=
	\begin{dcases}
		*\phi(x_0) -\frac{1}{{\rm Area}(M)}\int_M\phi,\quad &\text{if $M$ is closed},\\
		*\phi(x_0),\quad &\text{otherwise},
	\end{dcases}\\
	G_{x_0}&=G_{x},\\
	\md K(G_{x_0}) &=0\quad\text{on }\partial M.
\end{align} 
\end{rem}

Let us remember that $\Omega_{[r]}^1(M)$ is identified with $\vf^r(M)$ by the sharp operator $\sharp: \alpha\in  \Omega_{[r]}^1(M) \to \alpha_\sharp \in \vf^r(M)$, satisfying $\alpha[X]=g(\alpha_\sharp,X)$ for all $X\in\vf^r(M)$. 
As $1$-form generates, every $C^r$ $1$-current $T$ generates a vector field allowing singularities in $\sfs(T)$. 
For each $T\in {\cal D}_1'^r(M)$, we define $T_\sharp\in\vf^r(M\setminus \mathsf{S}^r(T))$ by $T_\sharp=(K(T))_\sharp$. 
As an example, for each $\psi\in \dcal_0'^r(M)$, we define $\jgrad \psi \in\vf^{r-1}(M\setminus \mathsf{S}^{r-1}(\md \psi))$ by $\jgrad \psi= (K(*\md \psi))_\sharp$. 
The vector field $\jgrad \psi$ stands for the Hamiltonian vector field induced by the Hamiltonian $\psi$ with singularities in $\sfs(\psi)$. 
We will use this notion when we take a vector field generated by point vortices. 
In what follows, for a given $T\in {\cal D}_p'^r(M)$, we abbreviate $\mathsf {S}^r(T)$ to $\mathsf {S}(T)$. 
Moreover, we denote $K(T)$ briefly by $T$ as long as no confusion arises.  
In particular, for a given $T\in\dcal_1'^r(M)$, when we write $|T|^2$ , it stands for not the multiplication of currents but the multiplication of the local $1$-form $|K(T)|^2$. 
Similarly, $(*\md T)*T$ means the multiplication of the local $0$-form $*\md K(T)$ and the local $1$-form $*K(T)$. 
This treatment is sensitive when we formulate the nonlinear term of the Euler-Arnold equation in the sense of currents. 
All currents are $C^r$ and the singular support consists of a finite set of points in this paper. 

We define the principal value $\pv:T\in \Omega_{[r]\rm loc}^p (M)\to \pv T\in {\cal D}_p'(M)$ by 
\begin{equation}
 \pv T[\phi]=\lim_{\varepsilon\to 0}\int_{M\setminus B_\varepsilon(\mathsf{S}(T))} T\wedge\phi
\end{equation}
for each $\phi\in {\cal D}^{2-p}(M)$ if the limit exists. 
The domain of $\pv$, say $\dom(\pv)$, is defined as the space of $p$-currents in which the limit exists for every $\phi\in{\cal D}^{2-p}(M)$. 
Obviously, $T=\pv T$ does not always hold true. 
As a matter of fact, if $T=\delta_p$, we have $\pv T= 0$. 

\begin{rem}
	Let us apply the calculus of currents to the derivation of the Biot-Savart kernel in the Euclidean plane $(\ce ,|\md z|^2)$. 
First, let us consider a $0$-current $\psi= \langle G_H,\delta_0 \rangle\in \dcal_0'(\ce)$. 
Since $\psi=I(G_H(z,0))$ on $\ce\setminus\{0\}$, we deduce that $\psi$ is $C^\infty$ and that $\mathsf{S}(\psi)=\{0\}$. 
We next compute $u=-*\md  \psi\in \dcal_1'^\infty (\ce)$. 
For each $\phi\in \dcal^1(\ce)$, 
\begin{align}
		u [\phi] 
		&= \md \psi[*\phi]
		=-\psi[\md *\phi]
		=-\int_{\ce\setminus\{0\} } G_H(z,0)\md *\phi\\
		&=\int_{\ce\setminus\{0\} }-\md(G_H(z,0)  *\phi)+ \md G_H(z,0) \wedge *\phi\\
		&=\int_{\ce\setminus\{0\} }-* \md G_H(z,0) \wedge \phi
		=I(-* \md G_H(z,0))[\phi],
\end{align}
which yields $u=I(-*\md G_H(z,0))$ in $\ce\setminus\{0\}$. 
It follows from $*\md z=-\mi\md z$ and $*\md \bar{z}=\mi\md \bar{z}$ that 
\begin{align}
	-*\md G_H(z,0) = \mi (\partial_z G_H(z,0) \md z-\partial_{\bar z} G_H(z,0)\md \bar{z}).
\end{align}
Therefore we obtain the Biot-Savart kernel as follows.
\begin{align}
	-\jgrad \psi 
	&= (-*\md G_H(z,0))_\sharp 
	= \mi (2\partial_z G_H(z,0) \partial_{\bar{z}}-2\partial_{\bar z} G_H(z,0)\partial_{z})\\
	&=\mi \left(
	-\frac{1}{2\pi z} \partial_{\bar{z}} 
	+\frac{1}{2\pi \bar{z}} \partial_z
	\right)
	=\frac{\mi }{2\pi |z|^2} (z \partial_z-\bar{z} \partial_{\bar{z}})\\
	&=\frac{1}{2\pi (x^2+y^2)} (-y\partial_x+x\partial_y).
\end{align}
\end{rem}

\begin{rem}
\label{rem:comp_vort}
We compute the vorticity of $-\jgrad\psi$, in which $\pv$ plays a key role in the computation by the density of $\psi$. 
Indeed, since $u=I(-*\md G_H(z,0))$ in $\ce\setminus\{0\}$, $*\md u=I(-*\md *\md G_H(z,0))=I(-\triangle G_H(z,0))=I(0)$ in $\ce\setminus\{0\}$. 
On the other hand, we have $*\md u= -*\md *\md \psi= -\triangle \psi=\omega=\delta_0$.
This indicates that we need to take the singular behavior into account in order to calculate the vorticity by the density.
For each $\phi\in \dcal^2(\ce)$, owing to $-*\md G_H(z,0)\in\oloc^1(\ce)$, we see that 
	\begin{align}
		*\md \pv u[\phi]
		&= \pv u[\md *\phi]
		=\lim_{\varepsilon\to 0} \int_{\ce\setminus B_\varepsilon(0) }-* \md G_H(z,0) \wedge \md*\phi\\
		&=\lim_{\varepsilon\to 0} \int_{\ce\setminus B_\varepsilon(0) }\md(*\phi*\md G_H(z,0)) -*\phi \md * \md G_H(z,0)\\
		&=\lim_{\varepsilon\to 0} -\int_{\partial B_\varepsilon(0) }*\phi*\md G_H(z,0) \\
		&-\lim_{\varepsilon\to 0}\int_{\ce\setminus B_\varepsilon(0) }*\phi \triangle G_H(z,0)\md z\wedge\md \bar{z}\\
		&=\lim_{\varepsilon\to 0} \int_{\partial B_\varepsilon(0) }*\phi\mi (\partial_z G_H(z,0) \md z-\partial_{\bar z} G_H(z,0)\md \bar{z})\\
		&=*\phi(0),
	\end{align}
	which yields $*\md \pv u=\delta_0$ in $\dcal_0'(\ce)$. 
Therefore, in order to recover the vorticity from the density, we need to consider $*\md \pv u$ instead of $*\md u$. 
\end{rem}

Let us introduce the operator $\loc$ by $\loc = \md\pv: T\in \Omega_{[r]\rm loc}^p(M) \to \loc T\in {\cal D}_{p+1}'(M)$, which is called the \textit{localizing operator}. 
For each $\phi\in{\cal D}^{p+1}(M)$, we have
\begin{equation} 
	\loc T[\phi] = \md\pv T [\phi] = (-1)^{p+1}\pv T[\md\phi]. 
\end{equation}
The domain of $\loc$, $\dom(\loc)$, is the space of $p$-currents  $T$ in which $\pv T[\md\phi]$ is well-defined for every $\phi\in{\cal D}^{1-p}(M)$. 
If $T\in \dom(\loc)$ satisfies $\md T=0$ in $M\setminus \sfs(T)$, then we obtain
\begin{align}
	\loc T[\phi]
	&=(-1)^{p+1}\pv T[\md \phi]
	=(-1)^{p+1}\lim_{\varepsilon\to 0}\int_{M\setminus B_\varepsilon(\mathsf{S}(T))} T \wedge \md \phi\\
	&=-\lim_{\eps\to 0} \int_{\partial B_\varepsilon(\mathsf{S}(T))}T\wedge\phi,
\end{align}
since $(-1)^{p+1} T  \wedge \md \phi=-\md(T\wedge\phi )+\md T\wedge\phi$. 
Hence, $\loc T$ is determined by the asymptotic behavior of $T$  near the singular support $\sfs(T)$. 
The name of \textit{localizing} operator is named after this property.

\subsection{Weak formulation of vector fields}
\label{sec:weak_vector_field}
Based on the fact that $\vf^r(M)$ is isomorphic to $\Omega_{[r]}^1(M)$ through the flat operator $\flat: v\in \vf^r(M)\to v^\flat=g(v,\cdotp)\in \Omega_{[r]}^1(M)$, 
we can obtain a weak extension of the notions associated with vector fields such as the divergence, the vorticity and the slip-boundary condition in the sense of currents by replacing the velocity form with a $1$-current. 
We will use these notions to formulate the Euler-Arnold equations in the sense of currents. 

As we see in Section~\ref{sec:backgound}, the divergence and the vorticity of $v\in\vf^r(M)$ is defined by $\delta v^\flat\in \Omega_{[r-1]}^0 (M)$ and $*\md v^\flat \in \Omega_{[r-1]}^0 (M)$. 
The slip boundary condition is $*X^\flat =0$ in $\omr^1(\partial M)$. 
Hence, it is reasonable to define the divergence and the vorticity of a $1$-current $\alpha\in \mathcal{D}_1'(M)$ by $\delta\alpha\in \mathcal{D}_0'(M)$ and $*\md \alpha\in \mathcal{D}_{0}'(M)$. 
On the other hand, for each $\alpha\in \dcal_1'^r(M)$ with $\partial M\cap \mathsf{S}(\alpha)=\emptyset$, we define the slip boundary condition as $*K(\alpha)=0$ in $\omr^1(\partial M)$.
Note that the slip boundary condition can not be defined for any $1$-current straightforwardly since the restriction of the current on the boundary does not in general make sense. 

\begin{rem}
\label{rem:Biot_Savart}
	In the plane, every incompressible vector field is a Hamiltonian vector field. 
	Hence, for any incompressible vector field $X\in \vf^r(\re^2)$, the vector field $X$ can be recovered from its vorticity $\omega$ by use of the Biot-Savart law $X= -\jgrad \langle G_H , \omega \rangle$. 
	However, for curved surfaces, even for a multiply connected domain in the plane, not every incompressible vector field becomes a Hamiltonian vector field in general owing to the Hodge decomposition. 
	With no restriction to Hamiltonian vector field, we can recover an incompressible vector field from the vorticity as follows. 
	Let $X\in \vf^r(M)$ be an incompressible vector field. 
	For a given incompressible vector field $Y\in \vf^r(M)$, let us assume that the relative vector field $Y-X$ is a Hamiltonian vector field: $Y-X= -\jgrad \psi$ for some $\psi\in \omr^0(M)$. 
	Then we note that the vector field $Y$ does not need to be a Hamiltonian vector field, but just the difference between $Y$ and $X$ need to be a Hamiltonian vector field. 
	The relative vorticity $\omega=*\md (Y-X)^\flat$ now satisfies 
\begin{equation}
	\omega = *\md (-\jgrad \psi)^\flat=-\triangle\psi. 
\end{equation}
	Hence, we obtain $Y=X-\jgrad \langle G_H,\omega\rangle$. 
\end{rem}

Let us extend the argument in Remark~\ref{rem:Biot_Savart} with respect to vector fields to currents by replacing $Y^\flat$ with $\alpha\in \dcal_1'^r(M)$. 
Let us fix incompressible vector field $X\in \vf^r(M)$ arbitrarily. 
For a given current $\alpha\in \dcal_1'^r (M)$ with $\partial M\cap \mathsf{S}(\alpha)=\emptyset$, let us assume that $\alpha$ satisfies the slip boundary condition and the relative current $\alpha-X^\flat $ is coexact, i.e., $\alpha- X^\flat=-*\md \psi$ for some $\psi\in {\cal D}_0'^r(M)$. 
Defining the relative vorticity $\omega\in  {\cal D}_0'^r(M)$ to $X$ by $\omega=*\md(\alpha-X^\flat)$, we obtain 
\begin{equation}
\label{eq:Poisson_equation}
 -\triangle \psi =-\delta\md\psi =*\md(\alpha-X^\flat)=\omega, 
\end{equation}
which gives $\alpha= X^\flat -*\md \langle G_H,\omega\rangle$.

In the present paper, we consider a special form of a singular vorticity as follows. 
\begin{ddef}
\label{def:singular_vorticity_point_vortices}
	Let $Q_N$ denote $({\rm Int}M)^N\setminus \{(q_n)_{n=1}^N\in ({\rm Int}M)^N|\,\exists i,j,\,q_i=q_j\}$. 
	Fix $N\in\ze$, $(\Gamma_n)_{n=1}^N \in (\mathbb{R}\setminus\{0\})^N$ and $(q_n)_{n=1}^N\in Q_N$. 
	A $0$-current $\omega\in {\cal D}_0'(M)$ is called a singular vorticity of point vortices placed on $\{q_n\}_{n=1}^N\subset M$, if for each $\phi\in {\cal D}^0(M)$, 
	\begin{equation}
		\omega[\phi]=\sum_{n=1}^N \Gamma_n *\phi(q_n)+c\int_M \phi,
	\end{equation}
	where 
	\begin{equation}
		c=
		\begin{dcases}
			-\frac{1}{{\rm Area}(M)}\sum_{n=1}^N \Gamma_n,\quad &\text{if $M$ is closed},\\
			0,\quad &\text{otherwise}.
		\end{dcases}
	\end{equation}
\end{ddef}
As we see in Remark~\ref{rem:Green}, the hydrodynamic Green function $G_H$ is written in terms of a current $G_{x_0}$. 
From this, every solution $\psi\in\dcal_0'^\infty(M)$ of the Poisson problem $-\triangle\psi=\omega$ is presented by  
\begin{equation}
\label{eq:corresponding_stream_function}
	\psi =\psi^0+ \sum_{n=1}^N \Gamma_n  G_{q_n}
\end{equation} 
up to a harmonic function $\psi^0$, 
since for each $\phi\in\dcal^2(M)$, 
\begin{align}
	-\triangle\psi [\phi]
	&=-\triangle\psi_0[\phi]+\sum_{n=1}^N\Gamma_n (-\triangle G_{q_n}[\phi])\\
	&=\sum_{n=1}^N\Gamma_n\left( *\phi(q_n)+\frac{1}{{\rm Area}(M)}\int_M \phi\right)\\
	&=\omega.
\end{align}
Identifying the delta function with the Dirac measure, we see that
\begin{equation}
	K(\omega) = c,\quad K(\psi)(p) = \psi^0(p)+ \sum_{n=1}^N \Gamma_n  G_H(p,q_n),
\end{equation}
which yields $\mathsf{S}(\omega)=\mathsf{S}(\psi )=\{q_n\}_{n=1}^N$. 

In the present paper, let us focus on the following vector field, which governs the evolution of point vortices. 
Let us fix $N\in\ze_{\ge1}$, $(\Gamma_n)_{n=1}^N \in (\mathbb{R}\setminus\{0\}) ^N$ and $(q_n)_{n=1}^N\in Q_N$. 
For each $n\in\{1,\ldots,N\}$, a vector field $v_n \in\vf^1( B_r(q_n))$ with sufficiently small $r\in \re_{>0}$ is defined by 
	\begin{equation}
		v_n(q) =  -\jgrad_q \left[\sum_{m=1}^N\Gamma_m G_H(q_m,q)+\frac{\Gamma_n}{2\pi}\log d(q_n,q) \right].
	\end{equation}
	It follows from the regularity theorem for a linear elliptic operator~\cite{Aubin_1998} that $v_n$ is $C^1$. 
	Note that, in particular, 
	\begin{equation}
		v_n(q_n) =  -\jgrad_{q_n} \left[\sum_{m\ne n}^N\Gamma_m G_H(q_m,q_n)+\Gamma_n R(q_n) \right].
	\end{equation}
	Since $\sum_{m\ne n}^N\Gamma_m G_H(q_m,q_n)+\Gamma_n R(q_n)$ is smooth on $Q_N$, $v:(q_n)_{n=1}^N\in Q_N\to (v_n(q_n))_{n=1}^N \in TQ_N$ is a smooth vector field on $Q_N$. 
	The point vortex dynamics is defined as a solution of the following ordinary differential equation,
	\begin{equation}
		\dot{q}_n(t)=v_n(q_n(t)),\quad n=1,\ldots,N.
	\end{equation}
	For a given Euler-Arnold flow $(X_t,P_t)\in\vf^r(M)\times C^r(M)$ and $(\beta_X,\beta_\omega)\in \re^2$, 
	The point vortex dynamics in the background field $X_t$ is defined as a solution of the following ordinary differential equation, called the point vortex equation,
	\begin{equation}
	\label{eq:point_vortex_equation}
		\dot{q}_n(t)=\beta_X X_t(q_n(t))+\beta_\omega v_n(q_n(t)),\quad n=1,\ldots,N.
	\end{equation}

\section{Euler-Arnold flow} 
\label{sec:Euler_Arnold}
Let us derive some equivalent presentations of the Euler-Arnold equation. 
They are derived from the following dual presentation. 
\begin{equation}
	\label{eq:dual_EA}
	\partial_t v^\flat +\nabla_{v} v^\flat =-\md p.
\end{equation}
Based on the fact that 
\begin{align}
	\nabla_v v^\flat
	&= \lie_v v^\flat-\md |v|^2/2,\label{eq:symm_advection}\\
	\nabla_v v^\flat
	&= i_v \md v^\flat+\md |v|^2/2,\label{eq:antisymm_advection}
\end{align}
we obtain two equivalent formulations of~\eqref{eq:dual_EA}, 
\begin{align}
	\partial_t v^\flat +\lie_v v^\flat -\md |v|^2/2&=-\md p,\label{eq:symm_EA}\\
	\partial_t v^\flat +i_v \md v^\flat +\md |v|^2/2&=-\md p,\label{eq:antisymm_EA}
\end{align}
where $\lie $ and $i$ are the Lie derivative and the interior multiplication respectively. 
Applying the curl operator $\curl= *\md $ to~\eqref{eq:symm_EA}, we have the vorticity equation. 
The vorticity equation contains contains both $v$ and $v^\flat$. 
On the other hand, in order to weakly formulate the Euler-Arnold equation in terms of current, we need to present the Euler-Arnold equation only in terms of $v^\flat$. 
We notice that 
\begin{align}
	i_v\md v^\flat=\omega i_v\dv =(*\md  v^\flat)*v^\flat,
\end{align}
owing to $\dim M=2$. 
In addition, let us recall the Riemannian metric can be extended to the metric on the cotangent bundle, which yields $|v^\flat|$ makes sense and $|v^\flat|=|v|$ holds.
Hence, from the equation~\eqref{eq:antisymm_EA}, we deduce that 
\begin{equation}
\label{eq:form_EA}
	\partial_t v^\flat+(*\md  v^\flat)* v^\flat+\md |v^\flat|^2/2 =-\md p.
\end{equation} 

Before we replace differential forms in~\eqref{eq:form_EA} with currents, we need to deal with the nonlinear term carefully in order to avoid multiplication of currents. 
Let us recall multiplication of local $p$-forms is still valid and a local $p$-form is converted to a current by taking the principle value. 
Hence, the Euler-Arnold equation is reformulated for $\alpha_t\in\dcal_1'^r(M)$ and $p_t\in\oloc^0(M)$ as follows. 
\begin{equation}
\label{eq:current_EA}
\partial_t \pv K(\alpha_t) + \pv \{K(*\md \alpha_t)K(* \alpha_t) +\md |K(\alpha_t)|^2/2\} =-\pv \md p_t,
\end{equation}  
in $\dcal_1'(M)$, if each of terms is contained in $\dom(\pv)$. 
When we focus on evolution of vorticity, the vorticity equation is useful rather than the Euler-Arnold equation. 
Let us remember that the vorticity equation is obtained by applying the curl operator $\curl=* \md$ to the Euler-Arnold equation. 
Based on this, applying the differential operator $\md$ to~\eqref{eq:current_EA}, we obtain the vorticity equation corresponding to~\eqref{eq:current_EA} if each density in~\eqref{eq:current_EA} is contained in $\dom(\md \pv)=\dom(\loc)$. 
For shorten notation, we use the same letter of a current for its density. 
\begin{ddef}
\label{def:weak_EA}
	A pair of time-dependent currents $(\alpha_t,p_t)\in \mathcal{D}_1'^r(M)\times \oloc^0(M)$ is called a weak Euler-Arnold flow, if the following conditions are satisfied: 
	\begin{enumerate}
		\item $\alpha_t\in\dom(\loc)$ and $(*\md \alpha_t)* \alpha_t +\md |\alpha_t|^2/2\in\dom(\loc);$
		\label{item:1st_condition_definition_weak_Euler_Arnold_flow}
		\item $\md p_t\in\dom(\loc);$
		\label{item:2nd_condition_definition_weak_Euler_Arnold_flow}
		\item $\partial_t \loc \alpha_t + \loc \left\{(*\md \alpha_t)* \alpha_t +\md |\alpha_t|^2/2\right\} =- \loc \md p_t\quad \text{in } \mathcal{D}_2'(M);$
		\label{eq:weak_EA} 
		\item $\delta \alpha_t = 0$ in $\mathcal{D}_1'(M);$ 
		\label{item:3rd_condition_definition_weak_Euler_Arnold_flow}
		\item $\partial M\cap \mathsf{S}(\alpha_t)=\emptyset$ and $* \alpha_t = 0$ on $\partial M$. 
		\label{item:4th_condition_definition_weak_Euler_Arnold_flow}
	\end{enumerate}
	In particular, we call the third condition the \textit{weak Euler-Arnold equation} and $\alpha_t$ the velocity current. 
\end{ddef}

Let us decompose a weak Euler-Arnold flow into a regular part and a singular part, thereby discussing the decomposition of each term in the weak Euler-Arnold equation (Definition~\ref{def:weak_EA}-\ref{eq:weak_EA}). 
Let us fix a weak Euler-Arnold flow $(\alpha_t,p_t)\in \mathcal{D}_1'^r(M) \times \oloc^0 (M)$ and a classical Euler-Arnold flow $(X_t,P_t)\in\vf^r(M)\times C^r(M)$: 
\begin{equation}
	\partial_t X_t^\flat+(*\md  X_t^\flat)* X_t^\flat+\md |X_t^\flat|^2/2 =-\md P_t.
\end{equation}
Writing $u_t=\alpha_t-X_t^\flat\in \dcal_1'^r(M)$, $\omega_t =*\md u\in\dcal_0'^r(M)$ and $\omega_{X}=*\md X_t^\flat$, for the corresponding densities, we obtain  
\begin{align}
	\partial_t \loc \alpha-\partial_t \loc X^\flat
	&=\partial_t \loc u\\
	\loc (*\md\alpha)*\alpha - \loc \omega_X* X^\flat 
	&= \loc (\omega+\omega_X)*(u+X^\flat ) - \loc \omega_X* X^\flat \\
	&=\loc\{(\omega_X +\omega) * u + \omega * X^\flat\},\\
	\loc \md |\alpha|^2 - \loc \md |X|^2 
	&= \loc \md g(X^\flat+u,X^\flat+u)-\loc\md  g(X^\flat,X^\flat)\\
	&= \loc \md g(2X^\flat +u, u). 
\end{align}
where we omit the subscript $t$. 
Hence, the weak Euler-Arnold equation (Definition~\ref{def:weak_EA}-3) is reduced to 
\begin{equation}
\label{eq:rel_weak_EA}
	\partial_t \loc u + \loc \{(\omega_X +\omega) * u + \omega * X^\flat+\md g ( 2X^\flat + u , u)/2\} = -\loc\md (p-P).
\end{equation}
In addition, as we see in Section~\ref{sec:weak_vector_field}, if $u_t$ is a coexact $1$-current for each time $t$, we obtain the Biot-Savart law for currents: $u_t=-*\md \langle G_H,\omega_t\rangle$. 
Then, the equation~\eqref{eq:rel_weak_EA} is regarded as an evolution equation for the relative vorticity $\omega$ and the relative pressure $p-P$. 

\begin{rem}
\label{rem:lem_plane}
	As an example, let us compute these terms in the equation~\eqref{eq:rel_weak_EA} in the case where the flow field is the Euclidean plane $(\ce,|\md z|^2)$ and $\omega$ comes from a singular vorticity of point vortices. 	
	Let us take a weak Euler flow $(\alpha_t,p_t)\in \mathcal{D}_1'^r(\ce) \times \oloc^0(\ce)$ and a classical Euler flow $(X_t,P_t)\in \vf^r(M)\times C^r(M)$. 
	Let us assume $u_t=\alpha_t-X_t^\flat\in \dcal_1'^r(M)$ is a coexact $1$-current for each time $t$ and the relative vorticity $\omega_t= *\md u_t$ is a singular vorticity of point vortices placed on $\{z_n(t)\}_{n=1}^N$ where 
	$(z_n(t))_{n=1}^N$ is an orbit of a $C^r$ one-parameter family $\Phi: t\in [0,T]\to \Phi_t\in {\rm Diff}^r(Q_N)$: $(z_n(t))_{n=1}^N =\Phi_t((z_n(0))_{n=1}^N)$.  
	
	Then, we now have $u=\alpha-X^\flat= -*\md \langle G_H,\omega\rangle$ and $K(\omega)=0$. 
	Hence, we obtain 
	\begin{align}
		\loc K(\omega)K(*u)=\loc K(\omega)K(*X^\flat)=0. 
	\end{align} 
	Let us show $\loc K(*u)=0$, thereby we see 
	\begin{align}
		|\loc \omega_XK(*u)[\phi] |\le ||\omega_X||_\infty |\loc K(*u)[\phi]| =0
	\end{align}
	and 
	\begin{align}
		\loc \{K(\omega_X +\omega) K(* u) + K(\omega) K(* X^\flat)\} = 0.
	\end{align} 
	Let us fix $\phi\in\dcal^0(\ce)$ and sufficiently small $\varepsilon>0$. 
	Defining $v_n\in \vf^\infty(\ce \setminus\{z_m\}_{m\ne n}^N)$ as  $v_n^\flat=K(u)-(\Gamma_n/2\pi) *\md \log|z-z_n|$, 
	we deduce from the Stokes theorem that 
	\begin{align}
		\int_{\ce  \setminus B_\eps (\sfs(u) )}K(*u) \wedge \md \phi 
		&=\int_{\ce  \setminus B_\eps (\sfs(u) )} -\md (\phi K(*u))+\phi \md K(*u)\\
		&=\sum_{n=1}^N \int_{\partial B_\varepsilon(z_n)}\phi K(*u)\\
		&=\sum_{n=1}^N \int_{\partial B_\varepsilon(z_n)}\phi\left( v_n^\flat-\frac{\Gamma_n}{2\pi}\md \log|z-z_n|\right)\\
		&=O(\varepsilon),
	\end{align}
	which yields $\loc K(*u)=0$. 
	In contrast, Proposition~\ref{prop:mean_value_theorem} yields that 
	\begin{align}
		\loc \md g ( X^\flat , K(u) )&= -\sum_{n=1}^N \Gamma_n \md \chi_{z_n}X,\\
		\loc \md g ( K(u) , K(u) )/2&= -\sum_{n=1}^N \Gamma_n \md \chi_{z_n}v_n,
	\end{align}
	since 
	\begin{align*}
		\int_{\ce \setminus B_\eps (\sfs (u) )} \md g ( X^\flat , K(u))\wedge\md \phi 
		&=-\sum_{n=1}^N \int_{\partial B_\eps (z_n)} g( X^\flat , K(u))\md \phi\\
		&=-\sum_{n=1}^N \int_{\partial B_\eps (z_n)} g(X,v_n^\flat)\md \phi\\
		&-\sum_{n=1}^N \int_{\partial B_\eps (z_n)} \frac{\Gamma_n}{2\pi} g( X, \jgrad \log |z|) \md \phi\\
		&\to -\sum_{n=1}^N \Gamma_n \md \chi_{z_n}X[\phi]\quad \text{as }\varepsilon \to 0,
	\end{align*}
	and 
	\begin{align}
		\int_{\ce \setminus B_\eps (\sfs (u) )} \md g ( K(u/2) , K(u))\wedge\md \phi 
		&=-\sum_{n=1}^N \int_{\partial B_\eps (z_n)} g( K(u/2) , K(u))\md \phi\\
		&=-\sum_{n=1}^N \int_{\partial B_\eps (z_n)}  |v_n|^2/2\md \phi\\
		&-\sum_{n=1}^N \int_{\partial B_\eps (z_n)} \frac{\Gamma_n}{2\pi} g( v_n, \jgrad \log |z|)  \md \phi\\
		&-\sum_{n=1}^N \int_{\partial B_\eps (z_n)} \frac12 \left(\frac{\Gamma_n}{2\pi}|*\md \log|z-z_n||\right)^2  \md \phi\\
		&\to -\sum_{n=1}^N \Gamma_n \md \chi_{z_n}v_n[\phi]\quad \text{as }\varepsilon \to 0. 
	\end{align}
\end{rem}

Remark~\ref{rem:lem_plane} illustrates that the leading terms in the advection term consists of $\loc\md g(X^\flat,u)$ and $\loc\md |u|^2/2$. 
It will be confirmed in Lemma~\ref{lem:decomposition_terms} that this property holds true for general curved surfaces. 
Based on this property, we introduce a model for the pressure that $\loc p$ is killed out with a linear combination of singular terms $\loc\md g(X^\flat,u)$ and $\loc\md |u|^2/2$, that is, 
\begin{equation}
	p=P+(2\beta_X-1)  g ( X^\flat, u)+ (2\beta_\omega -1)|u|^2/2
\end{equation}
for some $(\beta_X,\beta_\omega)\in\re^2$. 
This mathematical model can be interpreted that the singular behavior of the pressure is balanced with the interaction energy density $g ( X^\flat, u)$ and the kinetic energy density $|u|^2/2$ with a growth rate $(\beta_X,\beta_\omega)$. 
Summarizing the above, we propose the following regular-singular decomposition of a weak Euler-Arnold flow.
\begin{ddef}
\label{def:decomposable}
	 A weak Euler-Arnold flow $(\alpha_t,p_t)\in \mathcal{D}_1'^r(M)\times \oloc^0(M)$ is said to be $C^r$ decomposable $(r\ge1)$, if there exists a classical Euler-Arnold flow $(X_t,P_t)\in \vf^r(M)\times C^r(M)$ and $(\beta_X,\beta_\omega)\in\re^2$ 
	 such that the following conditions are satisfied for each time $t$. 
	 \begin{enumerate}
	 	\item $u_t=\alpha_t-X_t^\flat$ is coexact; 
	 	\item $p_t=P_t +(2\beta_X-1)  g ( X_t^\flat, u_t)+ (2\beta_\omega -1)|u_t|^2/2$. 
	 \end{enumerate}
	 Then we call $X_t$ a background field of $\alpha_t$, $\alpha_t-X_t$ a relative velocity current and $(\beta_X,\beta_\omega)$ a growth rate of $p_t$. 
\end{ddef}
Let us note that the $C^r$ decomposability of the weak Euler-Arnold flow guarantees the existence of the decomposition but there is no mention of the uniqueness of the decomposition. 
Hence, when we study a $C^r$ decomposable weak Euler-Arnold flow $(\alpha_t,p_t)\in \mathcal{D}_1'^r(M)\times \oloc^0(M)$, we need to fix a classical Euler-Arnold flow $(X_t,P_t)\in \vf^r(M)\times C^r(M)$ and a parameter $(\beta_X,\beta_\omega)\in\re^2$ such that the velocity field $X_t$ is a background field of $\alpha_t$ and the parameter $(\beta_X,\beta_\omega)$ is a growth rate of $p_t$. 
If a weak Euler-Arnold flow $(\alpha_t,p_t)$ is $C^r$ decomposable $(r\ge1)$, the equation \eqref{eq:rel_weak_EA} is written without the pressure term as follows.
\begin{equation}
\label{eq:beta_EA}
	\partial_t \loc u + \loc \{(\omega_X +\omega) * u + \omega * X^\flat+\md g ( 2\beta_X X^\flat + \beta_\omega u , u)\} = 0.
\end{equation}

\section{Main results}
\label{sec:main_results}
Let us fix $N,r\in \ze_{\ge1}$, $(\Gamma_n)_{n=1}^N\in \re^N$ and a $C^r$ one-parameter family $\Phi:t\in [0,T] \to \Phi_t\in {\rm Diff}^r(Q_N)$ in what follows. 
Let us denote by $(q_n(t))_{n=1}^N =\Phi_t((q_n(0))_{n=1}^N)$ an orbit of $\Phi$. 
We first prove that for a given $C^r$-decomposable weak Euler-Arnold flow, if the relative vorticity is given by a singular vorticity of point vortices placed on $\{q_n(t)\}_{n=1}^N$, $q_n(t)$ is a solution of the point vortex equation~\eqref{eq:point_vortex_equation}, which defines the point vortex dynamics in a background field.
\begin{thm}
\label{thm:decomposable_weak_Euler_Arnold_flow}
	Let $(\alpha_t,p_t)\in \mathcal{D}_1'^r(M)\times \oloc^0 (M)$ be a $C^r$-decomposable weak Euler-Arnold flow. 
	Fix a background field $X_t$ of $\alpha_t$, a growth rate $(\beta_X,\beta_\omega)$ of $p_t$. 
	Suppose the relative vorticity $\omega_t$ is a singular vorticity of point vortices placed on $\{q_n(t)\}_{n=1}^N$. 
	Then, $q_n(t)(n=1,\ldots,N)$ is a solution of the point vortex equation~\eqref{eq:point_vortex_equation}. 
\end{thm}
Conversely, we next prove that if $q_n(t)$ is a solution of the point vortex equation, there exists a $C^r$-decomposable weak Euler-Arnold flow such that the relative vorticity is given by a singular vorticity of point vortices placed on $\{q_n(t)\}_{n=1}^N$. 
\begin{thm}
\label{thm:existence_weak_Euler_Arnold_flow}
	Fix a classical Euler-Arnold flow $(X_t,P_t)\in \vf^r(M)\times C^r(M)$ and $(\beta_X,\beta_\omega)\in\re^2$. 
	Let $\omega_t\in {\cal D}_0'(M)$ be a singular vorticity of point vortices placed on $\{q_n(t)\}_{n=1}^N$. 
	Define a time-dependent current $u_t\in{\cal D}_1'^\infty (M)$ by $u_t=-*\md\langle G_H,\omega_t\rangle $. 
	Suppose $q_n(t)(n=1,\ldots,N)$ is a solution of the point vortex equation~\eqref{eq:point_vortex_equation}. 
	Then, the following pair of time-dependent currents $\alpha_t$ and $p_t$ defines a $C^r$-decomposable weak Euler-Arnold flow.
	\begin{align}
		\alpha_t &= X_t^\flat + u_t\in \mathcal{D}_1'^r(M),\\
		p_t&=P_t+(2\beta_X-1)g(X_t^\flat,u_t)+(2\beta_\omega-1)|u_t|^2/2\in\oloc^0(M). 
	\end{align}
\end{thm}
The following lemma plays a key role in the proofs of Theorem~\ref{thm:decomposable_weak_Euler_Arnold_flow} and~\ref{thm:existence_weak_Euler_Arnold_flow}. 
\begin{lem}
\label{lem:decomposition_terms}
	Fix a time-dependent vector field $X_t\in\vf^r(M)$ and $(\beta_X,\beta_\omega)\in \re^2$. 
	Let $\omega_t$ be a singular vorticity of point vortices placed on $\{q_n(t)\}_{n=1}^N$. 
	Define a time-dependent current $u_t\in{\cal D}_1'^\infty (M)$ by $u_t=-*\md\langle G_H,\omega_t\rangle $. 
	Then, we have
	\begin{align}
	&\partial_t \loc u = \sum_{n=1}^N \Gamma_n \md\chi_{q_n} \dot{q}_n\label{eq:1st_localize},\\
	&\loc \{(\omega_X +\omega) * u + \omega * X^\flat\} = 0\label{eq:2nd_localize},\\
	&\loc \{\md g ( 2\beta_X X^\flat + \beta_\omega u , u)\} = -\sum_{n=1}^N \Gamma_n \md\chi_{q_n} (\beta_X X+\beta_\omega v_n)\label{eq:3rd_localize}.\\
	\end{align}
	
	\end{lem}
\begin{proof}
	In what follows, since we fix $t\in[0.T]$, we may omit the subscript $t$ unless otherwise stated. 
	We decompose $u$ into the regular part $v_n^\flat$ and the singular part $*\md_q(\Gamma_n/2\pi) \log d(q,q_n)$ as follows.
	\begin{equation}
	\label{eq:decomposition_u}
		u=v_n^\flat+*\md_q\frac{\Gamma_n}{2\pi} \log d(q,q_n). 
	\end{equation}
	Without loss of generality, a geodesic polar coordinate $(\rho,\theta)=(d(q,q_n), \theta)$ can be taken in the neighborhood of $q_n$, satisfying  
	\begin{equation}
	\label{eq:geodesic_coordinate}
		*\md \log \rho= \md\theta + O(\rho ). 
	\end{equation}
	Let us fix $\phi\in\dcal^0(M)$ and sufficiently small $\eps>0$. 
	We remember the relation $\alpha \wedge \md \phi=-\md(\phi\alpha)+\phi\md\alpha$ for each $\alpha\in\Omega^1(M)$.
	To show \eqref{eq:1st_localize}, we see that $\loc u=\omega$. 
	By the Stokes theorem, we have 
	\begin{align}
		\int_{M\setminus B_\eps(\sfs(u) )} u\wedge\md\phi
		&=\int_{M\setminus  B_\eps(\sfs(u) )} -\md (\phi u)+\phi\md u\\
		&=\sum_{n=1}^N \int_{\partial B_\eps(q_n)} \phi u+ \int_{M\setminus  B_\eps(\sfs(\omega ) )} \omega*\phi.
	\end{align}
	It follows from \eqref{eq:decomposition_u} and \eqref{eq:geodesic_coordinate} that the first term becomes
	\begin{align}
		\int_{\partial B_\eps(q_n)} \phi u 
		&= \int_{\partial B_\eps(q_n)} (\phi(q)-\phi(q_n)+\phi(q_n)) \left(v_n^\flat+*\md_q\frac{\Gamma_n}{2\pi} \log \rho(q)\right)\\
		&=\frac{\Gamma_n}{2\pi}\phi(q_n)\int_{\partial B_\eps(q_n)}\md\theta+O(\eps)=\Gamma_n\phi(q_n)+O(\eps). 
	\end{align}
	Regarding the second term, we have 
	\begin{equation}
		\int_{M\setminus  B_\eps(\sfs(\omega ) )} \omega*\phi= c\int_{M\setminus  B_\eps(\sfs(\omega ) )} *\phi. 
	\end{equation}
	Thus we obtain 
	\begin{equation}
		\loc u[\phi]= \sum_{n=1}^N \Gamma_n\phi(q_n)+c\int_M*\phi=\omega[\phi],
	\end{equation}
	which yields~\eqref{eq:1st_localize}. 
	We compute each term in \eqref{eq:2nd_localize}. 
	Owing to $K(\omega)=c$, we have $\loc \omega*X^\flat =c\loc *X^\flat =c\md * X^\flat= 0$. 
	Since $\loc \omega*u=c\loc *u$ and $|\loc \omega_X*u[\phi]|
		\leq ||\omega_X|_{\supp \md \phi}||_{\infty}|\loc *u[\phi]|$, if we prove $\loc *u=0$, the assertion~\eqref{eq:2nd_localize} follows. 
	Indeed, we have 
	\begin{align}
		\loc *u[\phi]
		&=\lim_{\varepsilon\to 0} \int_{M\setminus  B_\eps(\sfs(u) )} *u \wedge \md \phi
		= \lim_{\varepsilon\to 0}\int_{M\setminus  B_\eps(\sfs(u) )} \{-\md(\phi \wedge*u) + \phi\md *u\}\\
		&= \sum_{n=1}^N \lim_{\varepsilon\to 0}\int_{\partial B_\eps(q_n)} \phi *u
		=\sum_{n=1}^N \lim_{\varepsilon\to 0}\int_{\partial B_\eps(q_n)} \phi \left(*v_n^\flat-\md_q\frac{\Gamma_n}{2\pi} \log \rho(q)\right)\\
		&=0.
	\end{align}
	It is easy to check \eqref{eq:3rd_localize} by Proposition~\ref{prop:mean_value_theorem}. 
	By Stokes theorem, we obtain 
	\begin{align}
		\int_{M\setminus B_\eps(\sfs(u) )}\md g ( 2\beta_X X^\flat, u)\wedge \md\phi 
		&= -\sum_{n=1}^N  \int_{\partial B_\eps(q_n)} g ( 2\beta_X X^\flat , u) \md\phi. 
	\end{align}
	It follows from the equation~\eqref{eq:decomposition_u} that 
	\begin{align}
		\int_{\partial B_\eps(q_n)} g ( 2\beta_X X^\flat , u) \md\phi
		&=\frac{\Gamma_n}{2\pi} \int_{\partial B_\eps(q_n)} g ( 2\beta_X X^\flat , *\md \log d(q,q_n) ) \md\phi +O(\eps)\\
		&=\frac{\Gamma_n}{2\pi} \int_{\partial B_\eps(q_n)} g ( 2\beta_X X , \jgrad \log d(q,q_n) ) \md\phi +O(\eps) \\
		&\to \Gamma_n\md \chi_{q_n}(\beta_X X)[\phi]\quad\text{as } \varepsilon\to 0. 
	\end{align}
	In the same manner, since 
	\begin{align}
		\int_{M\setminus B_\eps(\sfs(u) )}\md g (\beta_\omega u , u)\wedge \md\phi 
		&= -\sum_{n=1}^N \int_{\partial B_\eps(q_n)} g ( \beta_\omega u , u) \md\phi,
	\end{align}
	we see that 
	\begin{align}
		\int_{\partial B_\eps(q_n)} g ( \beta_\omega u , u) \md\phi
		&=\frac{\Gamma_n}{2\pi} \int_{\partial B_\eps(q_n)} g ( \beta_\omega v_n^\flat , *\md \log d(q,q_n) ) \md\phi \\
		&+\left( \frac{\Gamma_n}{2\pi}\right)^2 \int_{\partial B_\eps(q_n)} \beta_\omega |*\md \log d(q,q_n)|^2 \md\phi +O(\eps)\\
		&=\frac{\Gamma_n}{2\pi} \int_{\partial B_\eps(q_n)} g ( 2\beta_\omega v_n , \jgrad \log d(q,q_n) ) \md\phi +O(\eps) \\
		&=\Gamma_n\md \chi_{q_n}(\beta_\omega v_n)[\phi]\quad\text{as } \varepsilon\to 0,
	\end{align}
	which completes the proof. 
\end{proof}
We now show the two main theorems by using Lemma~\ref{lem:decomposition_terms}.
\begin{proof}[Proof of Theorem~\ref{thm:decomposable_weak_Euler_Arnold_flow}]
	Since the background field $X$ and the relative velocity current $u$ satisfy the assumptions of Lemma~\ref{lem:decomposition_terms}, the equalities ~\eqref{eq:1st_localize}-\eqref{eq:3rd_localize} hold true. 
	Since $X$ and $u$ come from a $C^r$-decomposable weak Euler-Arnold flow, they satisfy the equation~\eqref{eq:beta_EA}. 
	Substituting~\eqref{eq:1st_localize}-\eqref{eq:3rd_localize} into~\eqref{eq:beta_EA}, we obtain
	\begin{equation}
		\sum_{n=1}^N \Gamma_n \md\chi_{q_n} \left\{\dot{q}_n-(\beta_X X+\beta_\omega v_n)\right\}=0,
	\end{equation}
	which is the conclusion as desired. 
\end{proof}
\begin{proof}[Proof of Theorem~\ref{thm:existence_weak_Euler_Arnold_flow}]
	We first prove that the pair $(\alpha_t,p_t)$ is a weak Euler-Arnold flow. 
	Owing to Lemma~\ref{lem:decomposition_terms}, it is easy to check that $\alpha_t$ and $p_t$ satisfy the conditions in Definition~\ref{def:weak_EA} except for the weak Euler-Arnold equation. 
	In addition, we see that 
	\begin{equation}
		\partial_t \loc \alpha = \partial_t \loc (X^\flat +u) = \partial_t \md X^\flat+ \sum_{n=1}^N \Gamma_n \md\chi_{q_n} \dot{q}_n, 
	\end{equation}
	\begin{align}
		\loc \{(*\md \alpha)* \alpha \}
		&=\loc \{(*\md X^\flat)*X^\flat  \}+\loc \{(\omega_X +\omega) * u + \omega * X^\flat\} \\
		&= \md \{(*\md X^\flat)*X^\flat\}
	\end{align}
	and 
	\begin{align}
		\loc \md (|\alpha|^2/2 + p)
		&= \loc \md (|X^\flat|^2/2 + P)+\loc \{\md g ( 2\beta_X X^\flat + \beta_\omega u , u)\} \\
		&=-\sum_{n=1}^N \Gamma_n \md\chi_{q_n} (\beta_X X+\beta_\omega v_n). 
	\end{align}
	Since $(X,P)$ is an Euler-Arnold flow and $q_n$ is a solution of the point vortex equation~\eqref{eq:point_vortex_equation}, we deduce 
	\begin{align}
		\partial_t \loc \alpha_t + \loc \{(*\md \alpha_t)* \alpha_t \}+\loc \md (|\alpha_t|^2/2 + p_t)
		&=\partial_t \md X^\flat+\md \{(*\md X^\flat)*X^\flat\}\\
		&+\sum_{n=1}^N \Gamma_n \md\chi_{q_n} \left\{\dot{q}_n-(\beta_X X+\beta_\omega v_n)\right\}\\
		&=0,
	\end{align}
	which yields $(\alpha_t,p_t)$ is a weak Euler-Arnold flow. 
	By definition, it is obvious that the weak Euler-Arnold flow $(\alpha_t,p_t)$ is $C^r$-decomposable. 
\end{proof}

\section{Applications}
\label{sec:discussion}
As applications of these theorems, we now discuss two examples of point vortex dynamics in a background field: two identical point vortices in a linear shear in the Euclidean plane $(\ce ,\md z\md\bar{z})$ and $N$-point vortices on a surface in an irrotational flow.
 
Let us first check that the point vortex equation without any background field in the plane is obtained from our results as a special case. 
Let us set $\beta_X=0$ and $\beta_\omega=1$. 
Then, the point vortex equation~\eqref{eq:point_vortex_equation} is deduced as follows.
\begin{align}
\label{eq:PVE_usual}
	\dot{q}_n(t)=v_n(q_n). 
\end{align}
When the flow field is the plane, it follows from $G_H(z,z_0)=-(1/2\pi)\log |z-z_0|$ that for a given singular vorticity $\omega$ of point vortices placed on $\{z_n\}_{n=1}^N\subset \ce$,  
\begin{equation}
	u = -*\md \langle G_H,\omega\rangle = I\left( -*\md\psi_0-*\md\sum_{n=1}^N \Gamma_n  G_H(z,z_n)\right) 
\end{equation}
for some harmonic function $\psi_0$. 
We can choose $\psi_0=0$ without loss of generality. 
From $*\md z=-\mi\md z$ we deduce that 
\begin{align}
	-*\md G_H(z,z_n)
	&=\partial_z G_H(-*\md z)+\partial_{\bar{z}} G_H(-*\md \bar{z})\\
	&= \mi\partial_z G_H  \md z - \mi \partial_{\bar{z}}G_H \md \bar{z}.
\end{align}
Therefore, the dual vector field of $u$, denoted by  $u_\sharp=u^z\partial_z +u^{\bar{z}}\partial_{\bar{z}}\in \vf^\infty(M\setminus \mathsf{S}(u))$, can be written as 
\begin{equation}
	u^z(z) = \sum_{n=1}^N \Gamma_n(-2\mi \partial_{\bar{z}}G_H(z,z_n))
	=\frac{\mi }{2\pi}\sum_{n=1}^N \frac{\Gamma_n}{\bar{z}-\bar{z}_n}, 
\end{equation}
since $u_\sharp^\flat =u^z\md \bar{z} +u^{\bar{z}}\md z$. 
In the same manner, denoting $v_n(z)$ by $v_n(z)=v_n^z\partial_z +v_n^{\bar{z}}\partial_{\bar{z}}$, we obtain 
\begin{equation}
	v_n^z(z_n) = \sum_{\substack{m=1\\ m\ne n}}^N \Gamma_m(-2\mi \partial_{\bar{z}_n}G_H(z_n,z_m))
	=\frac{\mi }{2\pi}\sum_{\substack{m=1\\ m\ne n}}^N \frac{\Gamma_m}{\bar{z}_n-\bar{z}_m}. 
\end{equation}
Hence, we deduce from the point vortex equation~\eqref{eq:PVE_usual} to the following equation. 
\begin{align}
	\dot{z}_n=\frac{\mi }{2\pi}\sum_{\substack{m=1\\ m\ne n}}^N \frac{\Gamma_m}{\bar{z}_n-\bar{z}_m}.
\end{align}

\paragraph{Two identical point vortices in a linear shear}
Two identical point vortices in a linear shear is used as a model of the vortex merger in~\cite{Trieling_Dam_Heijst_2010}. 
The vortex merger is characterized as a fundamental process of the inverse cascade in $2$D turbulence. 
In the process, vortices with similarly small scales are affected by the shear flow induced from the surrounding vortices. 
As a result, small vortices combine to form a vortex with large scale. 
As a simple model for the vortex merger, two identical point vortices in a linear shear is adopted. 
To use the notations of this paper, let us set $N=2$, $\Gamma_1=\Gamma_2=\gamma$ and the linear shear $X=( cy,0)$ where $y=(z-\bar{z})/2\mi$. 
Based on~\cite{Trieling_Dam_Heijst_2010}, the evolution equation of two identical point vortices in a linear shear placed on $\{q_n(t)\}_{n=1}^2$ is given by
\begin{align}
\label{eq:evol_merger}
	\dot{q}_n(t)=X(q_n)+v_n(q_n). 
\end{align}
Since the equation~\eqref{eq:evol_merger} is a Hamiltonian system, plotting the Hamiltonian contours, we see that the topology of the contours changes around the threshold $\mu=c\xi_0^2/\gamma$ where $\xi_0\in \re$ is the initial distance between two point vortices. 
Obviously, for the case where $\mu=0$, there is no background field and two point vortices moves in a circle without changing the distance. 
When $\mu<0$, two point vortices are in periodic motion and the distance between them becomes shorter, 
which implies the vortex merger occurs. 
Otherwise, the distance is longer. 
This qualitative understanding is confirmed to be consistent with experimental studies in~\cite{Trieling_Dam_Heijst_2010}. 
In other words, it is pointed out that the sign of $\mu$, especially the orientation of the background field, determines the occurrence of vortex merger.

Let us apply the main theorems to this fact. 
First, Theorem~\ref{thm:existence_weak_Euler_Arnold_flow} shows that there exists a $C^\infty$-decomposable weak Euler flow $(\alpha_t,p_t)$  such that the relative vorticity is a singular vorticity of point vortices placed on $\{q_n(t)\}_{n=1}^2$. 
This guarantees the dynamics in~\cite{Trieling_Dam_Heijst_2010} comes from a weak Euler flow even though the evolution equation is formally derived with no relation with Euler flows. 
Second, as an application of Theorem~\ref{thm:decomposable_weak_Euler_Arnold_flow}, let us take a $C^\infty$-decomposable weak Euler flow $(\alpha_t,p_t)$ such that the background field is the linear shear $X$ and the growth rate is $(\beta_X,\beta_\omega)\in\re^2$ and the relative vorticity is a singular vorticity of point vortices placed on $\{q_n(t)\}_{n=1}^2$ with $\Gamma_1=\Gamma_2=\gamma$. 
From Theorem~\ref{thm:decomposable_weak_Euler_Arnold_flow}, we can deduce that $q_n(t)$ satisfies the point vortex equation~\eqref{eq:point_vortex_equation}. 
Then, we can easily see that $q_n(t)$ comes from a Hamiltonian system and the Hamiltonian contours are the same as those given in~\cite{Trieling_Dam_Heijst_2010} by changing parameters from $c$ to $c\beta_X$ and from $\gamma$ to $\gamma\beta_\omega$. 
In the similar manner, instead of $\mu$, we obtain the threshold $ \mu'=c\beta_X\xi_0^2/\gamma\beta_\omega$. 
Moreover, the same criterion is valid, that is, the vortex merger occurs if $\mu'<0$. 
Note that Theorem~\ref{thm:decomposable_weak_Euler_Arnold_flow} tells us   that the parameter $(\beta_X,\beta_\omega)$ is not just a dynamical parameter contained in the point vortex equation, but also a physical parameter derived from the pressure of the Euler flow.
From this point of view, we conclude that besides the orientation of the background field, the sign of the growth rate of the pressure determines the occurrence of vortex merger. 

\paragraph{$N$-point vortices on a surface in an irrotational flow.}
In the second case, let us take a $C^\infty$-decomposable weak Euler-Arnold flow $(\alpha_t,p_t)$ such that the background field is an irrotational field $X\in\vf^\infty(M)$ and the relative vorticity is a singular vorticity of point vortices placed on $\{q_n(t)\}_{n=1}^N$. 
As we see in Section~\ref{sec:Euler_Arnold}, the irrotational field is the fluid velocity of a steady Euler-Arnold flow $(X,P)\in\vf^\infty(M)\times C^\infty(M)$ and the pressure $P\in C^\infty(M)$ satisfies the Bernoulli law: $P=-|X|^2/2$. 
Let us fix the growth rate $(\beta_X,\beta_\omega)\in\re^2$ and the parameter $(\Gamma_n)_{n=1}^N\in (\re\setminus\{0\})^N$.
For simplicity, we ignore the interaction between the background field and point vortices, that is, we assume $\beta_X=0$. 
Then owing to Theorem~\ref{thm:decomposable_weak_Euler_Arnold_flow}, $q_n(t)$ and $p_t$ satisfy  
\begin{align}
\label{eq:usual_point_vortex_equation_with_pressure}
\begin{cases}
	\dot {q}_n = \beta_\omega v_n(q_n),\\ 
	p_t = -|\alpha_t|^2  + \beta_\omega|u_t|^2.
\end{cases}
\end{align}
Let us focus on two cases $\beta_\omega=0$ and $1$. 
In the case $\beta_\omega=0$, it follows from $\dot {q}_n(t)=0$ that each of point vortices does not move.  
Hence, we deduce that $(\alpha_t,p_t)$ is a steady solution. 
Moreover, the pressure satisfies a generalization of the Bernoulli law $p=-|\alpha|^2/2$ to the case with non-moving point vortices on curved surfaces. 
On the other hand, for $\beta_\omega=1$ we obtain the conventional  point vortex equation and the pressure satisfies a modified  Bernoulli law $p= -|\alpha|^2/2  +|u|^2$.
As a corollary, the Bernoulli law is generalized to the case where the flow field is a curved surface and where the presence of moving or non-moving point vortices is taken into account as follows.  
\begin{cor}[Generalized Bernoulli law with non-moving point vortices on surfaces]
	If a $C^\infty$-decomposable weak Euler-Arnold flow $(\alpha_t,p_t)$ on a surface satisfies that the background field of $\alpha_t$ is an irrotational field $X\in\vf^\infty(M)$ and that the pressure is given by $p_t=-|\alpha_t|^2/2$ and that the relative vorticity is a singular vorticity of point vortices, then $(\alpha_t,p_t)$ is a steady solution of the weak Euler-Arnold equations. 
\end{cor}
\begin{cor}[Modified Bernoulli law with moving point vortices on surfaces]
	If a $C^\infty$-decomposable weak Euler-Arnold flow $(\alpha_t,p_t)$ on a surface satisfies that the background field of $\alpha_t$ is an irrotational field and that $p_t=-|\alpha_t|^2/2 + |u_t|^2$ and that the relative vorticity is a singular vorticity of point vortices placed on $\{q_n(t)\}_{n=1}^N$, then for every $n\in \{1,\ldots,N\}$, $q_n(t)$ is a solution of the point vortex equation: 
	\begin{equation}
		\dot {q}_n(t)=v_n(q_n(t)).
	\end{equation}
\end{cor}
Let us finally discuss the role of the growth rate $\beta_\omega$ in the motion of point vortices and the pressure given in~\eqref{eq:usual_point_vortex_equation_with_pressure}. 
Denoting $Q_n(t)$ by the solution of $\dot {Q}_n=v_n(Q_n)$, the solution $q_n(t)$ in~\eqref{eq:usual_point_vortex_equation_with_pressure} can be written as $q_n(t)=Q_n(\beta_\omega t)$. Letting $\beta_\omega\to 0$, we see that $q_n$ moves quite slowly on the orbit of $Q_n$. 
In this sense, $\beta_\omega$ stands for the flexibility of the motion of point vortices besides the growth rate of the pressure relative to the kinetic energy density. 
We notice that as $\beta_\omega\to 0$ the pressure $p_t$ converges to $-|\alpha_t|^2/2 
$ in $\oloc^0(M)$, which is consistent with the generalized steady Bernoulli law. 
From this we can observe that point vortices are frozen if the pressure  $p_t$ is sufficiently close to the generalized Bernoulli law. 
As a consequence, we conclude that $\beta_\omega$ describes the the flexibility of the motion of point vortices and the growth rate of the pressure and that point vortices are slower to move as the pressure is sufficiently close to the generalized Bernoulli law.

\subsection*{Acknowledgements}
The author thanks to Professor Yoshihiko Mitsumatsu for giving opportunity to reconsider the necessity why point vortex dynamics should be formulated as a Hamiltonian system. 
The author acknowledges the helpful suggestions of Professor Takashi Sakajo. 
The author wishes to express gratitude to Professor Masayuki Asaoka in Doshisha University for several helpful comments concerning multiplication of currents. 
This work was supported by JSPS KAKENHI (no. 18J20037).

\end{document}